\documentclass[12pt]{article}

\usepackage{graphicx, amssymb, latexsym, amsfonts, amsmath, lscape, amscd,
amsthm, color, epsfig, mathrsfs, tikz, enumerate, multirow}
\usepackage{hyperref}
\usepackage{cleveref}

\newtheorem{theorem}{Theorem}[section]

\newtheorem{corollary}[theorem]{Corollary}

\newtheorem{lemma}[theorem]{Lemma}
\newtheorem{proposition}[theorem]{Proposition}

\newtheorem{observation}[theorem]{Observation}

\newtheorem{remark}[theorem]{Remark}
\newtheorem{claim}{Claim}[theorem]

\newcommand{\qedclaim}{\hfill $\diamond$ \medskip}
\newenvironment{proofclaim}{\noindent{\it Proof of Claim.}}{\qedclaim}
\newcommand\DELETE[1]{}

\newcommand{\R}{\small \mathcal{R}}
\newcommand{\C}{\small \mathcal{C}}

\begin{document}


\title{Cops and robber on variants of retracts and subdivisions of oriented graphs}
\author{
{\sc Harmender Gahlawat}$\,^{a}$, {\sc Zin Mar Myint}$\,^{b}$, {\sc Sagnik Sen}$\,^{b}$\\
\mbox{}\\
{\small $(a)$ Ben-Gurion University of the Negev, Beersheba, Israel}\\
{\small $(b)$ Indian Institute of Technology Dharwad, India.}\\}

\date{\today}

\maketitle

\begin{abstract}
\textsc{Cops and Robber} is one of the most studied two-player pursuit-evasion games played on graphs, where multiple \textit{cops}, controlled by one player, pursue a single \textit{robber}.  The main parameter of interest is the \textit{cop number} of a graph, which is the minimum number of cops that can ensure the \textit{capture} of the robber. 

\textsc{Cops and Robber} is also well-studied on directed/oriented graphs. In directed graphs, two kinds of moves are defined for players: \textit{strong move}, where a player can move both along and against the orientation of an arc to an adjacent vertex; and \textit{weak move}, where a player can only move along the orientation of an arc to an \textit{out-neighbor}. We study three variants of \textsc{Cops and Robber} on oriented graphs: \textit{strong cop model}, where the cops can make strong moves while the robber can only make weak moves; \textit{normal cop model}, where both cops and the robber can only make weak moves; and \textit{weak cop model}, where the cops can make weak moves while the robber can make strong moves. We study the cop number of these models with respect to several variants of retracts on oriented graphs and establish that the strong and normal cop number of an oriented graph remains invariant in their strong and distributed retracts, respectively. Next, we go on to study all three variants with respect to the subdivisions of graphs and oriented graphs. Finally, we establish that all these variants remain computationally difficult even when restricted to the class of 2-degenerate bipartite graphs. 
\end{abstract}





\section{Introduction}
Among the games modeled on graph search, the
two-player combinatorial pursuit-evasion game called {\sc Cops and Robber} is one of the most popularly studied in literature~\cite{bonato2011game}. The game was introduced independently by  
Quilliot~\cite{quilliot1978jeux}, and Nowakowski and Winkler~\cite{nowakowski1983vertex} on simple graphs. 
It gained a lot of popularity following its inception, primarily due to its various applications in topics like
artificial intelligence~\cite{isaza2008cover, klein2012catch}, constrained satisfaction problems and database theory~\cite{gottlob2000comparison,gottlob2001complexity}, distributed computing~\cite{d2017unified,czyzowicz2014evacuating} and network decontamination~\cite{flocchini2019distributed},
as well as for its deep impact on 
graph theory and algorithms~\cite{abraham2014cops,seymour1993graph}.
As a result, several variants of the game have been introduced and studied, many of which have deep connections and significant impacts on some of the aforementioned topics. For example, several variants of the game are shown to have correspondence with width parameters like treewidth~\cite{seymour1993graph}, pathwidth~\cite{parsons1}, tree-depth~\cite{depth}, hypertree-width~\cite{adler}, cycle-rank~\cite{depth}, and directed tree-width~\cite{dtwidth}.

Even though most of the variants are modeled on simple graphs, 
there exist natural variant(s) defined and studied on directed graphs and oriented graphs as well~\cite{loh2017cops,kinnersley2015cops,frieze2012variations}. Recently, Das et al.~\cite{das2021cops} studied three natural variations of the game on oriented graphs, namely, the \textit{strong}, \textit{normal}, and \textit{weak cop} models. In this article, we continue to build on their works by focusing on finding fundamental structural results for these models. We especially concentrate on exploring the
game's interaction with variants of retracts and particular types of subdivisions of the oriented graphs. Our structural results corresponding to the subdivisions also establish the computational hardness results for these variants. The primary goal of this paper is to contribute to building a theory of {\sc Cops and Robber} on oriented graphs.

\subsection{{\sc Cops and Robber} on Oriented Graphs}
An \textit{oriented graph} $\overrightarrow{G}$ is a directed 
graph having no loop or parallel arcs in opposite directions. An oriented graph may indeed have parallel arcs in the same direction 
between two vertices; however, for our works, such 
parallel arcs are redundant. 
Therefore, without loss of generality, unless otherwise stated, 
we will assume that the underlying graph $G$ of the oriented 
graphs $\overrightarrow{G}$ is simple, finite, connected, and contains at least two vertices.  With this, our  ``playing field'' (oriented graphs)  is ready, and thus, let us try to understand the game.

Let us assume $\overrightarrow{G}$ is the oriented graph on which we play the game. To begin with, Player~$1$ (the cop player) will place $k$ \textit{cops} on the vertices of $\overrightarrow{G}$, and then, Player~$2$ (the robber player) will place a \textit{robber} on a vertex of $\overrightarrow{G}$. After this initial set-up, the players will take turns, starting with Player~$1$, to move the cops (resp., the robber) from one vertex to another following game rules (depending on the game model). If, after a finite number of turns, a cop and the robber end up on the same vertex, that is, if a cop \textit{captures} the robber, then Player~$1$  wins. Otherwise, Player~$2$ wins.

This describes the game in general; however, the rules for moving the cops (resp., the robber) will be described while presenting the game models. On an oriented graph, two kinds of the move are of interest: a \textit{strong move} where a cop (or the robber) can shift from a vertex $u$ to its neighbor $v$ irrespective of the direction of the arc joining $u$ and $v$, and a
\textit{weak move} 
where a cop (or the robber) can shift from a vertex $u$ to its neighbor $v$ only if there is an arc from $u$ to $v$. 
The three models of the game are determined by the allowed moves for the Players~$1$ and~$2$. We list them below for convenience. 
\begin{enumerate}[(i)]
    \item \textit{The strong cop model:} In their turn, Player~$1$ can make at most one strong move for each of the cops, while Player~$2$ can only make at most one weak move for the robber. 

    \item \textit{The normal cop model:} In their turn, Player~$1$ (resp., Player~$2$) can make at most one weak move for each of the cops (resp. the robber). 

    \item \textit{The weak cop model:} In their turn, Player~$1$ can make at most one weak move for each of the cops, while Player~$2$ can make at most one strong move for the robber. 
\end{enumerate}
We need to recall some related necessary parameters~\cite{das2021cops} for continuing the study. 
The \textit{strong cop number} (resp., \textit{normal cop number}, 
\textit{weak cop number}) of an oriented graph $\overrightarrow{G}$, 
denoted by 
$c_s(\overrightarrow G)$ (resp., $c_n(\overrightarrow G)$, $c_w(\overrightarrow G)$), 
is the minimum number of cops needed by Player~$1$ to ensure a winning strategy on $\overrightarrow{G}$.  Moreover, $\overrightarrow{G}$ is \textit{strong-cop win} 
(resp., \textit{normal-cop win}, \textit{weak-cop win}) 
if its strong cop number 
(resp., normal cop number, weak cop number) is $1$. 
From the definitions, one can observe the relation:

\begin{equation}   
c_s(\overrightarrow G) \leq c_n(\overrightarrow G) \leq c_w(\overrightarrow G).
\end{equation}

Given a family $\mathcal{F}$ of oriented graphs, the parameters are defined by

\begin{equation}  
c_{\alpha}(\mathcal{F}) = \max\{c_{\alpha}(\overrightarrow G)\in \mathcal{F}\},
\end{equation}

for all $\alpha \in \{n, s, w\}$. 

\begin{remark}  If both Player~$1$ and Player~$2$ are allowed to make strong moves, then this game is the same as the game of \textsc{Cops and Robber} on the underlying undirected graph. Moreover, given an undirected graph $G$, its \textit{cop number}, denoted by $c(G)$, is the minimum number of cops needed by Player~$1$ to have a winning strategy for a game 
played on $G$. If the cop number of a graph $G$ is $1$, then we say that $G$ is cop win. 
\end{remark}

\subsection{Motivation and Context}
The normal cop model is well-studied in the context of directed/oriented graphs, while the two other variations are recent~\cite{das2021cops}. Hamidoune \cite{hamidoune} considered the game on Cayley digraphs. Frieze et al.~\cite{frieze2012variations} studied the game on digraphs and gave an upper bound of $\mathcal{O}\left(\frac{n(\log \log n)^2}{\log n}\right)$ for cop number in digraphs.

Hahn and MacGillivray~\cite{hahn} gave an algorithmic characterization of the cop-win finite reflexive digraphs and showed that any $k$-cop game can be reduced to $1$-cop game, resulting in an algorithmic characterization for $k$-copwin finite reflexive digraphs. 
However, these results do not give a structural characterization of such graphs. Later, Darlington et al.~\cite{darlington2016cops} tried to structurally characterize cop-win oriented graphs and gave a conjecture, which was later disproved by Khatri et al. ~\cite{khatri2018study}. This is evidence that the problem is not so straightforward to solve. 

Recently, the cop number of planar Eulerian digraphs and related families were studied in several articles \cite{mohar2,  hosseini2018, mohar}. Bradshaw et al.~\cite{bradshaw2021cops} proved that the cop number of directed and undirected Cayley graphs on abelian groups has
an upper bound of the form of $\mathcal{O}(\sqrt{n})$. Modifying this construction, they obtained families of graphs and digraphs with cop number $\Theta(\sqrt{n})$. 
In general, the problem of determining the cop number of a directed graph is known to be EXPTIME-complete due to Kinnersley~\cite{kinnersley2015cops}, which positively settled 
a conjecture by Goldstein and Reingold~\cite{goldstein1995complexity}.

Overall, the cop number is well-studied but, surprisingly, less understood on directed/oriented graphs. This article attempts to address this issue by studying some fundamentals in this domain.

\subsection{Our Contributions and Organization}
In Section~\ref{preliminaries}, we present some useful preliminaries.

In Section~\ref{retracts}, we deal with variants of retracts. To elaborate, the graph $G - v$ is a \textit{retract} of $G$ if there are vertices $u,v\in V(G)$ satisfying $N[v] \subseteq N[u]$. Here, we also say that $v$ is a \textit{corner vertex}. One key step in establishing the full characterization of cop win (undirected)  graphs was a lemma which proved that a graph is cop win if and only if its retract is also cop win. 
The characterization of weak-cop win oriented graphs also used a similar lemma for weak-retract (defined in Section~\ref{retracts}). We prove the analogs of the key lemmas for strong and normal models, even though we are yet to succeed in providing an exact characterization of  strong- (resp., normal-)cop win oriented graphs. 

In Section~\ref{strong-sub} and~\ref{weak-sub}, we study the effect of two different subdivisions, namely, the \textit{strong subdivision} and the \textit{weak subdivision}, on cop numbers.  The precise definitions are provided in Sections~\ref{strong-sub} and~\ref{weak-sub}, respectively. 
For classical {\sc Cops and Robber} game, some classical results study the effect of subdivisions on the cop number 
of an undirected graph establishing that the cop number of a graph does not decrease if we subdivide each of its edges a constant number of times~\cite{berarducci1993cop,joret}. 
On the other hand, in~\cite{das2021cops}, a special case of the strong subdivision was used as a tool to prove results and provide interesting examples. In this article, we study the effect of these two subdivisions on the cop numbers and establish the relation between cop number parameters involving these subdivisions.

In Section~\ref{com-compl}, we prove that unless $P=NP$, determining the strong, normal, and weak cop numbers are not polynomial-time solvable, even if we restrict the input graphs to the class of $2$--degenerate bipartite oriented graphs. 

In Section~\ref{conclusion}, we conclude the article, including the mention of some open problems.

\section{Preliminaries}\label{preliminaries}

This paper considers the game on oriented graphs whose underlying graph is simple, finite, and connected. Let $\overrightarrow{G}$ be an oriented graph whose underlying graph is $G$. We also say that $\overrightarrow G$ is an \textit{orientation} of $G$. 
Let $\overrightarrow{uv}$ be an arc of $\overrightarrow{G}$. 
We say that $u$ is an \textit{in-neighbor} of $v$ and $v$ is an \textit{out-neighbor} of $u$.
Let $N_{\overrightarrow{G}}^-(u)$ and $N_{\overrightarrow{G}}^+(u)$ denote the set of in-neighbors and out-neighbors of $u$, respectively.  Moreover, let $N_{\overrightarrow{G}}^+[v] = N^+(v) \cup \{v\}$ and $N_{\overrightarrow{G}}^-[v] = N^-(v) \cup \{v\}$. When it is clear from the context, by $N_{\overrightarrow{G}}(v)$ we denote $N_{\overrightarrow{G}}^+(v) \cup N^-(u)$, and by $N_{\overrightarrow{G}}[v]$ we denote $N_{\overrightarrow{G}}(v)\cup \{v\}$. Similarly, for an undirected graph $H$ and a vertex $v\in V(H)$, let $N_H(v)$ denote the set of neighbors of $v$ and let $N_H[v] = N_H(v)\cup \{v\}$.  Moreover, when it is clear from the context, to ease the presentation, we drop the subscript ${\overrightarrow{G}}$ (and $H$) from these notations. 

A vertex without any in-neighbor is a \textit{source}, and a vertex without any out-neighbor is a \textit{sink}. A vertex $v$ is said to be \textit{dominating} if $N^+[v] = V(\overrightarrow{G})$.
Let $v$ be vertex of $\overrightarrow{G}$ and $S$ is a subset of vertices of $\overrightarrow{G}$ (i.e., $S\subseteq V(\overrightarrow{G})$). Then, we say that $v$ is a {\em source in $S$} if $S\subseteq N^+[v]$. Moreover, we say that $|N^+(v)|$ is the \textit{out-degree} of $v$, $|N^-(v)|$ is the \textit{in-degree} of $v$, and $|N^+(v)|+|N^-(v)|$ is the \textit{degree} of $v$. An undirected graph $G$ is $k$-\textit{degenerate} if, for every induced subgraph $H$ of $G$, there is a vertex in $H$ with at most $k$-degree. An oriented graph is $k$-\textit{degenerate} if its underlying graph is $k$-degenerate.



\section{Retracts}\label{retracts}
Retracts are shown to have close relationships with the game of \textsc{Cops and Robber} on undirected graphs~\cite{berarducci1993cop}. In fact, the first characterization of cop win graphs used the concept of retracts~\cite{bonato2011game}. Moreover, the characterization of weak-cop win graphs is based on the notion of \textit{weak-retracts} (defined below). Thus, it makes sense to study the (strong/weak/normal) cop number of oriented graphs with respect to retracts.

Given an oriented graph $\overrightarrow{G}$, let $u$ and $v$  be two adjacent vertices satisfying $N[v] \subseteq N[u]$.
In such a scenario, the oriented graph $\overrightarrow{G} - v$  is a \textit{strong-retract} of $\overrightarrow{G}$. 
Given an oriented graph 
$\overrightarrow{G}$, let $\overrightarrow{u_1v},\overrightarrow{u_2v}, \cdots,\overrightarrow{u_pv}$  be $p$ arcs satisfying $N^+(v) \subseteq N^+(u_i)$, for each $i \in [p]$, 
and $N^-(v) \subseteq \bigcup_{i=1}^p N^-[u_i]$. 
In such a scenario, the oriented graph $\overrightarrow{G} - v$  is a \textit{distributed-retract} of $\overrightarrow{G}$. 
Given an oriented graph $\overrightarrow{G}$, let $\overrightarrow{uv}$ be an arc satisfying 
$N(v) \subseteq N^+[u]$. In such a scenario, the oriented graph $\overrightarrow{G} - v$  is a \textit{weak-retract} of $\overrightarrow{G}$.

In~\cite{das2021cops}, it was proved that an oriented graph is weak-cop win if and only if its weak-retract is weak-cop win. Here, we extend this result to prove that the strong and normal cop number of an oriented graph remains invariant in their strong-retracts and distributed-retracts, respectively. 

\begin{theorem}\label{th strong-retract}
Let $\overrightarrow{G}'$ be a strong-retract of $\overrightarrow{G}$. 
Then $c_s(\overrightarrow{G})=c_s(\overrightarrow{G}')$.
\end{theorem}

\begin{proof} 
Since $\overrightarrow{G}'$ is a strong-retract of $\overrightarrow{G}$, we may assume that 
$\overrightarrow{G}'= \overrightarrow{G} -v$, 
$u$ and $v$ are adjacent, and $N[v] \subseteq N[u]$. 

First suppose that $c_s(\overrightarrow{G})=k$. 
We will use the winning strategy of $k$ cops in $\overrightarrow{G}$ to get a winning strategy for $k$ cops in $\overrightarrow{G}'$ with the only difference being: whenever a cop, say $\C$, has to move to the vertex $v$, it moves to $u$ instead. 
Observe that $\C$ can make this move as $N[v] \subseteq N[u]$ and $\C$ can make strong moves. For the same reason, the next move of $\C$ will be as it would have been in the winning strategy for $\overrightarrow{G}$. That is, if in $\overrightarrow{G}$, $\C$ stays on $v$ or moves to some $w\in N(v)$, then, in $\overrightarrow{G}'$, $\C$ will stay on $u$ or move to $w$, respectively. The second instance is possible as $N(v) \subseteq N[u]$. Since the movement of $\R$ is restricted to $\overrightarrow{G}'$, $k$ cops will capture $\R$ after a finite number of moves using this strategy.

For the converse, suppose that $c_s(\overrightarrow{G}')=k$. Before going to the proof, we will define the \textit{image of the robber}, denoted by $I_{\R}$, a function from $V(\overrightarrow{G}) \to V(\overrightarrow{G}')$ as follows: 
\begin{equation*}
I_{\R}(x) =
\begin{cases} 
 x & \text{if $x\neq v$}, \\
u & \text{if $x=v$}.
   \end{cases}
\end{equation*}
We will use the winning strategy of $k$ cops in $\overrightarrow{G}'$ to get a winning strategy for $k$ cops in $\overrightarrow{G}$.  
Let us assume that the game is being played on $\overrightarrow{G}$. However, moves of $\R$  on 
$\overrightarrow{G}$ are emulated via the image function $I_{\R}(x)$ on $\overrightarrow{G}'$. The cops will move in 
$\overrightarrow{G}'$ to capture the image of the robber, and the cops on $\overrightarrow{G}$ will play the exact same moves (it is possible as $\overrightarrow{G}'$ is a subgraph of $\overrightarrow{G}$). At the time of the capture of the
image of the robber in $\overrightarrow{G}'$, if the robber is on any vertex other than $v$ in $\overrightarrow{G}$, then it gets captured there as well. 
If the robber is on the vertex $v$ of $\overrightarrow{G}$ at the time when its image gets captured on $\overrightarrow{G}'$,  then in $\overrightarrow{G}$, there is a cop on the vertex $u$ at that point of time. Therefore, as $N[v] \subseteq N[u]$, the robber will get captured in the next move. 
\end{proof}


 Next, we show that the normal cop number of an oriented graph remains invariant in its distributed retracts. 

\begin{theorem}\label{th distributed-retract}
Let $\overrightarrow{G}'$ be a distributed-retract of $\overrightarrow{G}$. 
Then $c_n(\overrightarrow{G})=c_n(\overrightarrow{G}')$.
\end{theorem}

\begin{proof}
Since $\overrightarrow{G}'$ is a distributed-retract of $\overrightarrow{G}$, we may assume that 
$\overrightarrow{G}'= \overrightarrow{G} -v$, 
and 
$\overrightarrow{u_1v},\overrightarrow{u_2v}, \cdots,\overrightarrow{u_pv}$  are $p$ arcs satisfying $N^+(v) \subseteq N^+(u_i)$ 
and $N^-(v) \subseteq \bigcup_{i=1}^p N^-[u_i]$, for each $i \in [p]$.

First of all, suppose that $c_n(\overrightarrow{G})=k$. 
We are going to show that $k$ cops have a winning strategy in 
$\overrightarrow{G}'$ as well. The idea is to play the game simultaneously on $\overrightarrow{G}$ and $\overrightarrow{G}'$. The robber $\R$ will originally move in $\overrightarrow{G}'$, while on $\overrightarrow{G}$ it will simply mimic the moves (it is possible as $\overrightarrow{G}'$ is a subgraph of $\overrightarrow{G}$). 

On the other hand, the cops will use a winning strategy to capture the robber in $\overrightarrow{G}$, and we will use this strategy to provide a winning strategy on $\overrightarrow{G}'$ as well. 
In fact, we will use the 
exact same strategy on $\overrightarrow{G}'$ with the only difference being: a cop $\C$ will move to one of the $ u_i$s 
in $\overrightarrow{G}'$ when its corresponding cop moves to the vertex $v$ in $\overrightarrow{G}$. 
The choice of this $u_i$ will depend on the movement of $\C$ in that particular turn. To elaborate, 
in $\overrightarrow{G}$, $\C$ must have moved from some 
$v^- \in N^-(v)$ to $v$. According to the definition of distributed-retract, $v^{-}$ belongs to 
$N^{-}[u_j]$ for some $j\in [p]$. Choose the minimum index $i$ among all such $u_j$s for which $v^{-} \in N^{-}[u_j]$. The corresponding $u_i$ is our choice for positioning $\C$ in that particular turn. 
Observe that it is possible for a cop to make its moves following the above strategy in $\overrightarrow{G}'$, following the game rules. Since the movement of the robber $\R$ is restricted to the vertices of $\overrightarrow{G}'$, it will get captured on both graphs in this strategy.

Next, we will show the other direction, that is, we will
suppose that $c_n(\overrightarrow{G}')=k$ and show that there is a winning strategy for $k$ cops on $\overrightarrow{G}$. 
To do so, we will play the game simultaneously on $\overrightarrow{G}$ and $\overrightarrow{G}'$. The robber 
$\R$ will originally move on $\overrightarrow{G}$ and its shadow
$\R_S$ will move on $\overrightarrow{G}'$. Now, the $k$
cops will capture $\R_S$ on $\overrightarrow{G}'$ (as we know 
$k$ cops have a winning strategy on $\overrightarrow{G}'$). We 
will mimic the moves of the cops on $\overrightarrow{G}'$  on 
$\overrightarrow{G}$ (it is possible since 
$\overrightarrow{G}'$ 
is a subgraph of $\overrightarrow{G}$).   

To begin with, let us describe the movements of $\R_S$. 
Whenever $\R$ is at any vertex other than $v$, $\R_S$ is also 
on that vertex. If $\R$ starts at $v$, then $\R_S$ will start 
at $u_1$. 
Moreover, during the play, 
if $\R$ moves from a particular vertex $v^{-}$ to $v$, 
then $\R_S$ will move to $u_i$, 
where $i$ is the minimum index satisfying $v^- \in N^-[u_i]$.
Observe that, it is possible for $\R_S$ to make its moves 
following the above-mentioned rules.

After a finite number of moves, $\R_S$ will get captured on $\overrightarrow{G}'$. 
At that point in time, either $\R$ also gets captured on $\overrightarrow{G}$, or it must be placed on $v$ with a cop $\C$ placed on $u_j$ for some $j$. 
In the latter case, $\R$ will get captured in the next turn. 
\end{proof}

In particular, the above result implies that cop win oriented graphs are distributed-retract invariant. 
To complement the above result, we prove a sufficient condition for an oriented graph to be not  cop win.

\begin{theorem}
     If for every arc $\overrightarrow{uv}$ in $\overrightarrow{G}$, there exists an 
     out-neighbor $v^+$ of $v$ that is not an out-neighbor of $u$, then $\overrightarrow{G}$ is not cop win. 
\end{theorem}

\begin{proof}
    Suppose the cop $\C$ is \textit{attacking} the robber $\R$. That means, we may assume that $\C$ is on $u$ and $\R$ is on $v$ for some arc $\overrightarrow{uv}$. 
    We know that there 
    exists some $v^{+} \in N^+(v) \setminus N^+(u)$. Thus, the robber will move to such a $v^+$ and avoid the capture. 
\end{proof}

\section{Strong Subdivisions}\label{strong-sub}
Let $G$ be a simple, connected, and finite graph. Then, $\overrightarrow{S}_t(G)$ is the oriented graph obtained by replacing each edge $uv$ of $G$ by two directed paths of \textit{length} (number of arcs) $t$: one starting from $u$ and ending at $v$, and the other starting at $v$ and ending at $u$. The 
oriented graph $\overrightarrow{S}_t(G)$ is called the 
\textit{strong $t$-subdivision} of $G$. See Figure~\ref{fig:1} for a reference. As we deal only with simple oriented graphs here, the value of $t$ is at least $2$. 

\begin{figure}
    \centering
    \includegraphics{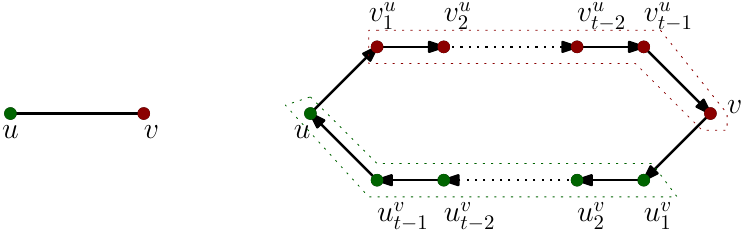}
    \caption{Illustration of subdivision of an edge $uv$. Here, in the subdivided part, for each red vertex $x$, $f(x)=v$, and for each green vertex $y$, $f(y)=u$.}
    \label{fig:1}
\end{figure}

For the ease of presentation of proofs, we provide an explicit construction of $\overrightarrow {S}_t(G)$ from $G$. Consider an edge $uv\in E(G)$. This edge is replaced by two directed paths of length $t$ each: one from $u$ to $v$ of the form $u v^u_1 v^u_2 \cdots v^u_{t-1} v$, and one from $v$ to $u$ of the form $v u^v_1 u^v_2 \cdots u^v_{t-1} u$. Moreover, the vertices $u$ and $v$ are termed as the \textit{original vertices}, and the vertices of the form $v^u_i$ and $u^v_j$ are termed as the \textit{new vertices}. Furthermore, we define a function $f:V(\overrightarrow {S}_t(G)) \rightarrow V(G)$ such that for any $x,y\in V(G)$, $i\in [t-1]$ and $x^y_i\in V(\overrightarrow {S}_t(G))$, $f(x^y_i)= x$ and $f(x) =x$. Finally, we have the following easy observation that will be useful for us. 

\begin{observation}\label{O:trivial}
For any two vertices $x,y\in V(\overrightarrow {S}_t(G))$, if there is a directed path from $x$ to $y$ of length at most $t$, then $f(y) \in N_G[f(x)]$ and $f(x)\in N_G[f(y)]$.
\end{observation}

In what follows, we will provide both upper and lower bounds on the strong, normal, and weak cop number of $\overrightarrow{S}_t(G)$ in terms of $c(G)$. In~\cite{das2021cops}, it was proved that the cop number of a graph $G$ is a natural lower bound for the strong cop number of $\overrightarrow{S}_2(G)$. Here, we generalize this result to $\overrightarrow{S}_t(G)$. Specifically, we have the following lemma. 



\begin{lemma}\label{L:lowerBound}
    Let $G$ be a simple graph. Then, for any $t>1$, $c_s(\overrightarrow{S}_t(G)) \geq c(G)$.
\end{lemma}

\begin{proof}
    Let $c_s(\overrightarrow{S}_t(G)) =k$. We will show that $k$ cops have a winning strategy in $G$ as well. To this end, we borrow the winning strategy of cops from $\overrightarrow{S}_t(G)$. As the game is played in $G$, we play a game simultaneously in $\overrightarrow{S}_t(G)$ to use the winning strategy of cops from $\overrightarrow{S}_t(G)$. 
    
    \medskip 
    
    \noindent \textit{Game Setup:} The $k$ cops begin by placing themselves on the vertices of $\overrightarrow{S}_t(G)$ as per the winning strategy. Accordingly, we place the cops on the vertices of $G$ such that if a cop, say $\C_i$, is placed on a vertex $x\in V(\overrightarrow{S}_t(G))$, then we place $\C_i$ on the vertex $f(x)$ in $V(G)$. Now, $\R$ enters a vertex, say $u$, in $V(G)$. In $\overrightarrow{S}_t(G)$ also, we place $\R$ on the same vertex $u$. 

    \medskip 
    
    \noindent \textit{Move Translations:} Now, the game proceeds as follows. Each round in $G$ is translated to $t$ rounds in $\overrightarrow{S}_t(G)$. For each move of $\R$ in $G$ from a vertex $u$ to a vertex $v$, we make $t$ moves of $\R$ in $\overrightarrow{S}_t(G)$ from $u$ to $v$ along the directed path $u v^u_1 \cdots v^u_{t-1} v$ if $u$ and $v$ are distinct vertices, else $\R$ stays at the same vertex for the next $t$ moves in $\overrightarrow{S}_t(G)$. Cops will move according to the winning strategy in $\overrightarrow{S}_t(G)$ for these $t$ moves in $\overrightarrow{S}_t(G)$. Notice that in these $t$ moves, if a cop starts from a vertex $x$ and finishes at a vertex $y$, then there is a directed path either from $x$ to $y$ or from $y$ to $x$ of length at most $t$. Therefore, due to Observation~\ref{O:trivial}, $f(y)\in N_G[f(x)]$. Thus, when a cop, say $\C_i$, moves from a vertex $x$ to a vertex $y$ in these $t$ moves in $\overrightarrow{S}_t(G)$, we move $\C_i$ from $f(x)$ to $f(y)$ in $G$. 
    
    \medskip
    
    \noindent \textit{Capture:} Then, the game goes on like this. Since $k$ cops have a winning strategy in $\overrightarrow{S}_t(G)$, they will capture $\R$ after a finite number of rounds in $\overrightarrow{S}_t(G)$. That is, after a finite number of rounds a cop, say $\C_i$, and $\R$ will be on the same vertex $x$ in $\overrightarrow{S}_t(G)$. This translates to $\C_i$ and $\R$ both being on $f(x)$ in $G$. This completes our proof.
\end{proof}

Next, we provide an upper bound on the normal cop number (and hence, on the strong cop number) of $\overrightarrow{S}_t(G)$ by establishing that it is at most $c(G)+1$. In particular, we have the following lemma.

\begin{lemma}\label{L:upperBound}
    Let $G$ be a simple graph. Then, $c_n(\overrightarrow{S}_t(G)) \leq c(G)+1$.
\end{lemma}

\begin{proof}
    Let $k$ cops have a winning strategy in $G$. We will use this strategy to get a winning strategy for $k+1$ cops in $\overrightarrow{S}_t(G)$ for the normal cop model. Here also, we will play two games simultaneously. 
    
    \medskip
    
    \noindent \textit{Game Setup:} The game begins with $k$ cops placing themselves on the vertices of $G$ as per the winning strategy. Let the $k$ cops be marked as $\C_1, \ldots, \C_k$. We place $k$ cops on the same vertices in $\overrightarrow{S}_t(G)$, i.e., if a cop $\C_i$ is on a vertex $x$ in $G$, then $\C_i$ is placed on the vertex $x$ in $\overrightarrow{S}_t(G)$ as well. Moreover, we have an extra \textit{dummy cop}, denoted $D_1$, in $\overrightarrow{S}_t(G)$. We place $D_1$ on an arbitrary vertex in $\overrightarrow{S}_t(G)$. Now, $\R$ enters on a vertex, say $x$, in $\overrightarrow{S}_t(G)$. We place $\R$ on $f(x)$ in $G$. 

    \medskip
    
    \noindent \textit{Move Translations:} Now, the game proceeds as follows. The cops move in $G$ as per their winning strategy. This move of cops is translated to $t$ moves of cops in $\overrightarrow{S}_t(G)$ as follows. If a cop $\C_i$ moves from a vertex $u$ to $v$ in $G$, then it moves from $u$ to $v$ in $\overrightarrow{S}_t(G)$ along the directed path $u v^u_1 \cdots v^u_{t-1} v$. During these $t$ moves, $\R$ might move from its current position $x$ to some vertex $y$ such that there is a directed path from $x$ to $y$ of length at most $t$. Therefore, due to Observation~\ref{O:trivial}, $f(y)\in N_G[f(x)]$. Thus, we move $\R$ from $f(x)$ to $f(y)$ in $G$. Then, the game goes on.

    \medskip
    
    \noindent \textit{Capture:} Since $k$ cops have a winning strategy in $G$, after a finite number of rounds, they can capture $\R$ in $G$. Consider the position of cops and the robber just before the capture. Let $\R$ be on a vertex $u\in V(G)$. Then, there is a cop at a vertex $v\in N_G(u)$, and for every vertex $w\in N_G(u)\setminus \{v\}$, there is a cop on some vertex in $N_G[w]$.  This position translates to $\overrightarrow{S}_t(G)$ in the following manner. $\R$ occupies either the vertex $u$ or some vertex $u^w_j$ where $w\in N_G(u)$ and $j<t$. Moreover, for every vertex $v\in N_G(u)$, there is a cop that can reach $v$ in at most $t$ rounds. Now, comes the role of the dummy cop. In each round, $D_1$ moves towards $\R$ (it can do so since $\overrightarrow{S}_t(G)$ is strongly connected) and forces $\R$ to move after every finite number of rounds. This way, first, $\R$ is forced to  move to $u$, and then to move to some vertex $v^u_1$ where $v\in N_G(u)$. Since there is a cop that can reach $v$ in at most $t$ rounds, this cop, say $\C_i$, will start moving towards $v$. Moreover, since $\R$ has to move after every finite number of rounds (due to the dummy cop), it will eventually reach $v$ in no more than $t-1$ rounds, where it will be captured by $\C_i$. Hence, $c_n(\overrightarrow{S}_t(G)) \leq c(G)+1$.   
\end{proof}

Thus, we get the following theorem as a consequence of Lemma~\ref{L:lowerBound} and Lemma~\ref{L:upperBound} that bounds both strong cop number and normal cop number of $\overrightarrow{S}_t(G)$ in terms of $c(G)$.

\begin{theorem}\label{th strong vs cn ub lb}
Let $G$ be a simple graph. Then, 
$c(G)\leq c_s(\overrightarrow{S}_t(G))\leq c_n(\overrightarrow{S}_t(G)) \leq c(G)+1$.
\end{theorem}

Theorem~\ref{th strong vs cn ub lb} also establishes a lower bound on the weak cop number of $\overrightarrow{S}_t(G)$. In the following result, we establish an upper bound on the weak cop number of $\overrightarrow{S}_2(G)$ in terms of $c(G)$. In particular, we have the following result.

\begin{theorem}\label{T:weakBound}
    Let $G$ be a simple graph. Then, $c(G) \leq c_w(\overrightarrow{S}_2(G)) \leq c(G)+2$.
\end{theorem}
\begin{proof}
    The lower bound follows directly from Theorem~\ref{th strong vs cn ub lb} by taking $t=2$. To prove the upper bound, we provide a winning strategy using $c(G)+2$ weak cops against a strong robber. Let $c(G)=k$. Here also, we will play the two games simultaneously.
    
    \medskip

    \noindent \textit{Game Setup:} The game begins with $k$ cops placing themselves on the vertices of $G$ as per the winning strategy. Let the $k$ cops be marked as $\C_1, \ldots, \C_k$. We place $k$ cops on the same vertices of $\overrightarrow{S}_2(G)$, i.e., if a cop $\C_i$ is placed on a vertex $x$ in $G$, then $\C_i$ is placed on the same vertex $x$ in $\overrightarrow{S}_2(G)$ as well. Moreover, we have two extra \textit{dummy cops}, $D_1$ and $D_2$. Now, $\R$ enters a vertex, say $u$, in $\overrightarrow{S}_2(G)$. Then, we place the robber on $f(u)$ in $G$.

    \medskip

    \noindent \textit{Move Translations:} Now, the game proceeds as follows. Each round in $G$ is translated to two rounds in $\overrightarrow{S}_2(G)$. Each round in $G$ begins with the $k$ cops moving as per their winning strategy. This move is translated to two moves of the cop player in $\overrightarrow{S}_2(G)$ as follows: If a cop $\C_i$ moves from a vertex $u$ to a vertex $v$ in $G$, then $\C_i$ moves from $u$ to $v$ in $\overrightarrow{S}_2(G)$ along the directed path $u v^u_1 v$ in two moves. During these two moves of the cops, $\R$ can also make two moves in $\overrightarrow{S}_2(G)$. 
    
    Let $\R$ moves from a vertex $u$ to a vertex $v$, and then to a vertex $w$ ($u, v,$ and $w$ are not necessarily distinct). If one of $v$ or $w$ is an original vertex in $\overrightarrow{S}_2(G)$, then $\R$ moves to that vertex in $G$, otherwise $\R$ does not move in $G$. Notice that $v$ and $w$ both cannot be original vertices if they are distinct.  Hence, at most one of them is an original vertex, and therefore, the next move of $\R$ in $G$ is well defined. Moreover, observe that, following this procedure, after a round, if $\R$ is at a vertex $u$ in $G$, then $\R$ is at a vertex in $N_{\overrightarrow{S}_2(G)}[u]$ in $\overrightarrow{S}_2(G)$.    
    
    \medskip

    \noindent  \textit{Capture:} Since $k$ cops have a winning strategy in $G$, they will be able to capture $\R$ in $G$ after a finite number of rounds. Let $\R$ be at a vertex $u$ just before the capture. At this instance, there is a cop at a vertex $v\in N_G(u)$, and for every vertex $w\in N_G(u)\setminus \{v\}$, there is a cop on some vertex in $N_G[w]$. Consider the translation of this situation in $\overrightarrow{S}_2(G)$. $\R$ is at a vertex in $\{u\} \cup N^+(u) \cup N^-(u)$. Now, we have the following claim.

    \begin{claim}\label{C:weak}
        If $\R$ is at the original vertex $u$ at this instance, it will be captured after a finite number of rounds.
    \end{claim}
\begin{proofclaim}
First, we establish that $\R$ will have to move after every finite number of rounds. To see this, let the robber occupies a vertex, say $y$, and the dummy cop $D_2$ occupies a vertex, say $z$. If $\R$ does not make a move, then $D_2$ moves towards $y$ along a shortest path between $y$ and $z$ in $\overrightarrow{S}_2(G)$. Note that such a path always exists and has a finite length since $\overrightarrow{S}_2(G)$ is a strongly connected finite digraph. Thus, if $\R$ does not move from $y$, $D_2$ will eventually reach $y$ and capture $\R$ in a finite number of rounds.

Now, if $\R$ moves to vertex $u^v_1$, then the cop at $v$ will capture $\R$. If $\R$ moves to the vertex $v^u_1$, then the cop at $v$ stays at its current vertex, and the dummy cop $D_2$ keeps moving towards the vertex $u$. Since the vertex $v$ is occupied by a cop, $\R$ cannot move to $v$ (as long as a cop occupies $v$). Moreover, since $D_2$ is moving towards $u$, observe that if $\R$ keeps oscillating between $v^u_1$ and $u$, it will be captured after a finite number of rounds. Hence, after a finite number of rounds, $\R$ will have to move to a vertex $w^u_1$ or $u^w_1$. At this point, the cop at $v$ moves to $u^v_1$, ensuring that $\R$ cannot return to $u$ in the next move. Moreover, recall that there is a cop, say $\C_1$, at a vertex in $N_G[w]$. Now, $\C_1$ will move towards $w$. Now, since $\R$ is at $w^u_1$ or $u^w_1$, and it has to move every finite number of rounds, it will have to move to either $u$ or $w$, and it cannot move to $u$ due to the cop at $u^v_1$. Hence, $\R$ will have to move to $w$ where it will be captured by $\C_1$. \end{proofclaim}

Due to Claim~\ref{C:weak}, we can assume that $\R$ is at a new vertex in $N^+(u) \cup N^-(u)$. Now again, $D_2$ will move towards $\R$ and will force $\R$ to move after a finite number of rounds. At this point, $\R$ is at an original vertex $x\in N_G[u]\setminus \{v\}$ and there is a cop at a vertex $y$ in $N_G(x)$ (in case $\R$ does not get captured at this step). Without loss of generality, let us rename the cops such that the cop at $y$ is named $D_1$. Next, we have the following claim.

    \begin{claim}\label{C:w2}
        Let $\R$ be at an original vertex $u$ in $\overrightarrow{S}_2(G)$ and $D_1$ be at an original vertex $v$ such that $v\in N_G(u)$. Then, in a finite number of rounds, $D_1$ and $D_2$ can force $\R$ to move to a vertex $w\in N_G(u) \setminus \{v\}$ such that $D_1$ is at the vertex $u$.
    \end{claim}
\begin{proofclaim}
    If $\R$ does not move, $D_2$ moves towards $\R$  and forces $\R$ to move after a finite number of rounds. Similarly to the arguments presented in the proof of Claim~\ref{C:weak}, the cops force $\R$ to first move to a vertex $w^u_1$ or $u^w_1$ such that $w\in N_G(u) \setminus \{v\}$. At this point, $D_1$ moves to $u^v_1$ ensuring that $\R$ cannot come back to $u$. Since $\R$ is forced by $D_2$ to move after every finite number of rounds, it will have to move to $w$ after a finite number of rounds, and $D_1$ moves to $u$ at this point.    
\end{proofclaim}
        
Now, the game proceeds as follows. The $k$ cops, $\C_1,\ldots, \C_k$, begin in $G$ again as per their winning strategy in $G$. Now, the move translation is slightly different. For each round in $G$, the $k$ cops move as per their winning strategy. In $\overrightarrow{S}_2(G)$, the cops move as per the move translation explained above in two rounds. In these two rounds, if $\R$ has moved to an original vertex then we move $\R$ in $G$ accordingly. Otherwise, $D_1$ and $D_2$ force $\R$ to move to an original vertex, as per Claim~\ref{C:w2}. The other cops, $\C_1,\ldots,\C_k$ stay at their current position during these moves in $\overrightarrow{S}_2(G)$. Once $\R$ has moved to an original vertex in $\overrightarrow{S}_2(G)$, we move $\R$ in $G$ according to the move translation. In this case, observe that when $\R$ gets captured in $G$, it is at an original vertex in $\overrightarrow{S}_2(G)$. Therefore, due to Claim~\ref{C:weak}, it gets captured in $\overrightarrow{S}_2(G)$ as well after a finite number of rounds. This completes our proof.
\end{proof}

Here, we present an easy result that will be used to prove winning strategies for cops.
\begin{lemma}\label{L:trivialCapture}
Let $G$ be a simple graph. Consider the strong cop model in the graph $\overrightarrow{S}_t(G)$. If on a cop's move, $\R$ occupies a vertex of the form $v^u_i$, where $i \in [t-1]$, and there is a cop $\C$ occupying a vertex $x$ such that there is a directed path from either $x$ to $v$ or $v$ to $x$ of length at most $(t-i)+1$, then $\R$ will be captured in at most $2t$ rounds. 
\end{lemma}
\begin{proof}
    Assume the scenario to be as in the statement. $\C$ begins by moving towards $v$ to decrease its distance to $v$ to $t-i$. Notice that a strong cop can move against the orientation of the arcs as well. Now, in the next $t-i$ rounds, the only possible moves for $\R$ are either staying at the same vertex or moving towards $v$ (since $\R$ can only make weak moves). Hence, after these $t-i$ moves, $\R$ is either at the vertex $v$ or at a vertex $v^u_j$ such that $j\geq i$. During these rounds, in each move, $\C$ moves towards $v$ and reaches $v$ after $t-i$ rounds. In this instance, if $\R$ is at the vertex $v$, it gets captured. Otherwise, $\R$ is at a vertex $v^u_j$ where $j\geq i$. Moreover, the only moves possible for $\R$ are to move towards $v$. Now, $\C$ can make strong moves towards $v^u_j$ and capture $\R$ in at most $t$ more rounds.
\end{proof}

Next, we provide a sufficient condition under which the cop number of $G$ and the strong cop number of its strong-subdivision is the same. 

\begin{theorem}\label{th triangle_free}
    Let $G$ be a triangle-free undirected graph. Then $c(G)=c_s(\overrightarrow{S}_t(G))$. 
\end{theorem}
\begin{proof}
    Let $c(G)=k$. Due to Theorem~\ref{th strong vs cn ub lb}, to establish our claim, it is sufficient to prove that $k$ strong cops have a winning strategy in $\overrightarrow S_t(G)$, i.e., $c_s(\overrightarrow{S}_t(G))\leq k$. We provide such a strategy below. 

    Similarly to the proof of Lemma~\ref{L:upperBound}, we play two games simultaneously in $G$ and $\overrightarrow S_t(G)$ and use the winning strategy of $k$ cops in $G$ to achieve the following in $\overrightarrow S_t(G)$. Each cop $\C_i$ is on an original vertex, $\R$ occupies a vertex $u$ or a vertex of the form $u^w_j$ (where $j<t$). Moreover, there is a cop, say $\C$ on a vertex $v\in N_G(u)$, and for each vertex $x\in N_G(u)$, there is a cop on some vertex $y\in N_G[x]$ in $\overrightarrow S_t(G)$. Now, we distinguish the following two cases.

    \begin{enumerate}
        \item $\R$ occupies the vertex $u$: In this case, $\C$ moves towards $u$ following the directed path $u v^u_{1} \cdots v^u_{t-1} v$ against the orientation of the arcs. Now, if $\R$ does not move for the next $t$ rounds, then it will be captured by $\C$ in at most $t$ rounds. If $\R$ starts moving towards $v$ (using the same path), observe that it will be captured after at most $t$ rounds.  If $\R$ starts moving towards some vertex $w\in N(u)\setminus \{v\}$, then observe that $\R$ is at the vertex $w^u_1$ and there is a cop at a distance at most $t$ from $w$. Hence, due to Lemma~\ref{L:trivialCapture}, $\R$ will be captured after at most $2t$ cop moves.

        \item $\R$ occupies a vertex $u^x_i$, then $\C$ starts moving towards $u$ following the directed path $u v^u_{1} \cdots v^u_{t-1} v$ against the orientation of the arcs. If $\R$ does not move to reach $u$ before $\C$ reaches $u$, then $\R$ will be captured in the next $t$ moves. If $\R$ reaches $u$ when $\C$ is at some vertex $v^u_{j}$, then notice that we have considered a similar scenario in the previous case. Hence $\R$ will be captured in at most $2t$ rounds.
    \end{enumerate}

    Thus, $k$ strong cops will capture $\R$ in $\overrightarrow{S}_t(G)$ after a finite number of rounds.  
\end{proof}

Even though strong-cop win characterization is an open problem, we can characterize all strong-cop win oriented graphs that are strong-subdivisions. 

\begin{theorem}\label{th SG strong-cn copwin characterization}
Let $G$ be a graph. Then $\overrightarrow{S}_t(G)$ is strong-cop win if and only if $G$ is a tree. 
\end{theorem}
\begin{proof}
    In one direction, let $G$ be a tree. Then, we know that $c(G)=1$. Moreover, since $G$ is triangle-free, we have $c_s(\overrightarrow{S}_t(G))= c(G)=1$ due to Theorem~\ref{th triangle_free}. Thus, $\overrightarrow{S}_t(G)$ is strong-cop win.

    In the reverse direction, we show that if $G$ is not a tree, then $c_s(\overrightarrow{S}_t(G))>1$. If $c(G) >1$, then $c_s(\overrightarrow{S}_t(G))>1$ due to Lemma~\ref{L:lowerBound}. Therefore, we assume $c(G)=1$ and look for a contradiction. Next, we have the following easy claim, which we prove for the sake of completeness.

    \begin{claim}\label{C:trivial}
        Let $G$ be a cop win graph that is not a tree. Then, $G$ contains at least one triangle. 
    \end{claim}
\begin{proofclaim}
Let $v$ be a corner vertex in $G$. Then $G$ is cop win if and only if $G-\{v\}$ is cop win~\cite{bonato2011game}. Therefore, let $H$ be the graph that we get by removing the \textit{leaves} (a vertex of degree 1) from $G$ recursively and exhaustively. Since each leaf is a corner, we have that $H$ is also a cop win graph. 

Since $G$ was not a tree, $H$ contains at least one cycle. Now, since $H$ is cop win, it must contain at least one corner vertex, say $u$. Let $u$ be a corner vertex of $v$, i.e., $N[u]\subseteq N[v]$. Now, since $u$ has degree at least two, $u$ has a neighbor $w$ distinct from $v$. Finally, since $N[u] \subseteq N[v]$, we have that $uvw$ is a triangle in $H$ as well as in $G$.    
\end{proofclaim}         

    Hence, due to Claim~\ref{C:trivial}, we have that $G$ contains at least one triangle, say $uvw$. Now, consider the graph $\overrightarrow{S}_t(G)$. The robber will stay in the subgraph of $\overrightarrow{S}_t(G)$ corresponding to the triangle $uvw$. The game begins with the cop, say $\C$, placing itself on a vertex of the graph. Irrespective of the beginning position of $\C$, there is at least one vertex $x\in \{u,v,w\}$ such that $x$ is not an in-neighbor or an out-neighbor of the current position of $\C$. Now, $\R$ stays on this vertex $x$ unless $\C$ moves to a vertex $x^y_{t-1}$ or a vertex $y^x_1$ where $y\in N_G(x)$. At this instance, notice that there is at least one vertex $z\in \{u,v,w\} \setminus \{x,y\}$.  Now, $\R$ moves to the vertex $z^x_1$, and in the next $t-1$ moves, keeps moving towards $z$. Now, we claim that $\C$ cannot capture $\R$ in these $t$ moves. It is because each vertex in the directed path between $x$ and $z$ is closer to $\R$ than $\C$. Finally, observe that once $\C$ reaches $z$, we are in a situation identical to the one we started. Hence, $\R$ will wait at the vertex $z$ unless it is under attack and then move to some other vertex of the triangle $uvw$. This way $\R$ will evade the capture forever. This completes our proof.
\end{proof}

\section{Weak Subdivisions}\label{weak-sub}
Let $\overrightarrow{G}$ be an oriented graph. Let $\overrightarrow{W}_t(\overrightarrow{G})$ be the oriented graph obtained by replacing each arc $\overrightarrow{uv}$ of $\overrightarrow{G}$ by a length $t$ directed path from $u$ to $v$ of the form $u v^u_1 \cdots v^u_{t-1} v$. The 
oriented graph $\overrightarrow{W}_t(\overrightarrow{G})$ is called the \textit{weak $t$-subdivision} of $\overrightarrow{G}$. Similarly to the definition of strong subdivisions, the vertices $u$ and $v$ are termed as \textit{original} vertices while the vertices of the form $u^w_j$ ($j<t$) are termed as \textit{new} vertices. Moreover, similar to Section~\ref{strong-sub}, we define a function $g:V(\overrightarrow{W}_t(\overrightarrow{G})) \rightarrow V(\overrightarrow{G})$ such that for any $x,y \in V(\overrightarrow{G})$, $i\in [t-1]$ and $x^y_i \in V(\overrightarrow{W}_t(\overrightarrow{G}))$, we have $g(x^y_i) = x$ and $g(x) = x$. Finally, we have an observation similar to Observation~\ref{O:trivial}. 

\begin{observation}\label{O:wtrivail}
    For any two vertices $x,y\in V(\overrightarrow{W}_t(\overrightarrow{G}))$, if there is a directed path from $x$ to $y$ of length at most $t$, then $g(y) \in N_{\overrightarrow{G}}^+[g(x)]$.
\end{observation}

We know that, given a simple graph $G$, its cop number does not decrease if we subdivide each edge of $G$~\cite{berarducci1993cop}. We prove its oriented analog for each of the three models. First, we have the following lemma.

\begin{lemma}\label{L:wnDivision}
    Let $\overrightarrow{G}$ be an oriented graph. Then, $c_n(\overrightarrow{W}_t(\overrightarrow{G})) \geq c_n(\overrightarrow{G})$ and  $c_w(\overrightarrow{W}_t(\overrightarrow{G})) \geq c_w(\overrightarrow{G})$.
\end{lemma}
\begin{proof}
Here we will play two games simultaneously in $\overrightarrow{G}$ and $\overrightarrow{W}_t(\overrightarrow{G})$. The proof of $c_n(\overrightarrow{W}_t(\overrightarrow{G})) \geq c_n(\overrightarrow{G})$ and $c_w(\overrightarrow{W}_t(\overrightarrow{G})) \geq c_w(\overrightarrow{G})$ are almost identical. We provide proof of $c_n(\overrightarrow{W}_t(\overrightarrow{G})) \geq c_n(\overrightarrow{G})$ first.

Let $c_n(\overrightarrow{W}_t(\overrightarrow{G})) =k$. We will use the strategy of these $k$ cops in $\overrightarrow{W}_t(\overrightarrow{G})$ to get a strategy for $k$ cops in $\overrightarrow{G}$. 

\medskip
\noindent\textit{Setup:} The game starts with $k$ cops placing themselves on the vertices of $G$ as per the winning strategy. Now, if a cop $\C_i$ is placed at a vertex $x$ in $\overrightarrow{W}_t(\overrightarrow{G})$, then we place $\C_i$ on $g(x)$ in $G$. Next, $\R$ enters on a vertex in $\overrightarrow{G}$. Then, we place $\R$ in the same vertex in $\overrightarrow{G}$.

\medskip
\noindent \textit{Move Translation:} Now, the game proceeds as follows. To ease the presentation, let the cops skip their first move in $\overrightarrow{G}$ as well as in $\overrightarrow{W}_t(\overrightarrow{G})$, and then each round in $\overrightarrow{G}$ consists of $\R$ moving, followed by cops moving. Note that it does not hurt the winning strategy of $k$ cops in $\overrightarrow{W}_t(\overrightarrow{G})$ since the cops can win irrespective of the starting position of $\R$ in $\overrightarrow{W}_t(\overrightarrow{G})$.

Then, each round in $\overrightarrow{G}$ is translated to $t$ rounds in $\overrightarrow{W}_t(\overrightarrow{G})$. If $\R$ moves from a vertex $u$ to $v$ in $\overrightarrow{G}$, then $\R$ moves from $u$ to $v$ in $\overrightarrow{W}_t(\overrightarrow{G})$, along the path $u v^u_1 \cdots v^u_{t-1} v$, in next $t$ rounds. In these $t$ rounds, cops move in $\overrightarrow{W}_t(\overrightarrow{G})$ as per their winning strategy. Note that each cop $\C_i$ moves from a vertex $x$ to a vertex $y$ such that there is a directed path from $x$ to $y$ in $\overrightarrow{W}_t(\overrightarrow{G})$ of length at most $t$. Hence, due to Observation~\ref{O:wtrivail}, $g(y) \in N_{\overrightarrow{G}}^+[g(x)]$. Thus, we move $\C_i$ from vertex $g(x)$ to $g(y)$ in $\overrightarrow{G}$.

\medskip
\noindent\textit{Capture:} Since $k$ cops have a winning strategy in $\overrightarrow{W}_t(\overrightarrow{G})$, they will capture $\R$ in $\overrightarrow{W}_t(\overrightarrow{G})$ after a finite number of rounds. Notice that at this point, $\R$ gets captured in $\overrightarrow{G}$ as well. 

This completes the proof of $c_n(\overrightarrow{W}_t(\overrightarrow{G})) \geq c_n(\overrightarrow{G})$. Now, the proof of $c_w(\overrightarrow{W}_t(\overrightarrow{G})) \geq c_w(\overrightarrow{G})$ is similar with the following changes. Let $c_w(\overrightarrow{W}_t(\overrightarrow{G})) = k$. Then we similarly borrow the strategy of $k$ cops in $\overrightarrow{W}_t(\overrightarrow{G})$ to get a winning strategy in $\overrightarrow{G}$. Here the setup is exactly the same. Next, in a move, if $\R$ makes a strong move from $u$ to $v$ in $\overrightarrow{G}$, then note that $\R$ can move from $u$ to $v$ in $\overrightarrow{W}_t(\overrightarrow{G})$ by making $t$ strong moves. Finally, the cops will move in $\overrightarrow{W}_t(\overrightarrow{G})$ according to the winning strategy and notice that when $\R$ gets captured in $\overrightarrow{W}_t(\overrightarrow{G})$, it gets captured in $\overrightarrow{G}$ as well. 
\end{proof}

In the following lemma, we prove that the strong cop number of an oriented graph does not decrease by operation of weak subdivisions.
\begin{lemma}\label{L:swDivision}
  Let $\overrightarrow{G}$ be an oriented graph. Then, $c_s(\overrightarrow{W}_t(\overrightarrow{G})) \geq c_s(\overrightarrow{G})$.  
\end{lemma}
\begin{proof}
    The proof here is similar to the proof of Lemma~\ref{L:wnDivision}. Here also, we will play two games simultaneously: one on $\overrightarrow{W}_t(\overrightarrow{G})$ and one on $\overrightarrow{G}$ in the following manner. Let $c_s(\overrightarrow{W}_t(\overrightarrow{G}))= k$.

    \medskip
    \noindent\textit{Setup:} The game begins with $k$ cops placing themselves on the vertices of $\overrightarrow{W}_t(\overrightarrow{G})$ as per their winning strategy. Now, if a cop $\C_i$ is placed on a vertex $x\in V(\overrightarrow{W}_t(\overrightarrow{G}))$, then we place $\C_i$ on a vertex of $\overrightarrow{G}$ in the following manner. If $x$ is an original vertex, then we place $\C_i$ on $x$. Else, $x$ is of the form $u^v_j$ (where $j < t$), and then we place $\C_i$ on $u$. (The choice of $u$ and $v$ is not important here. We can choose either of them, and the rest of the proof will remain the same.) Then, $\R$ enters on a vertex, say $w$ of $\overrightarrow{G}$. Then, we place $\R$ on $w$ in $\overrightarrow{W}_t(\overrightarrow{G})$. 

    \medskip
    \noindent\textit{Move Translation:} Now, the game proceeds as follows. The cops miss their first move as in the proof of Lemma~\ref{L:wnDivision}. Hence, we may assume that each round contains first the move of $\R$ and then the move of cops. Hence $\R$ moves in $\overrightarrow{G}$ from a vertex, say $u$, to a vertex, say $v$. This move gets translated to $t$ moves of $\R$ in $\overrightarrow{W}_t(\overrightarrow{G})$. Now,  $\R$ moves in $\overrightarrow{W}_t(\overrightarrow{G})$ from $u$ to $w$ along the directed path between $u$ and $w$ of length $t$. In these $t$ rounds, the cops move as per their winning strategy. Now, these $t$ moves of a cop $\C_i$ in $\overrightarrow{W}_t(\overrightarrow{G})$ are translated to a move of $\C_i$ in $\overrightarrow{G}$ in the following manner. Note that if $\C_i$ moves from a vertex, say $x$, to a vertex, say $y$, then there is a directed path from either $x$ to $y$ or from $y$ to $x$ of length at most $t$. If both $x$ and $y$ are original vertices, then $\C_i$ is at the vertex $x$ in $\overrightarrow{G}$ and it moves to the vertex $y$ in $\overrightarrow{G}$ (note that $x$ and $y$ are adjacent in $\overrightarrow{G}$). Otherwise, the $x$ to $y$ path in $G$ contains at most one original vertex. If it does not contain any original vertex, then $\C_i$ stays at the same vertex, else $\C_i$ moves to the original vertex contained in the path, say $z$, in $\overrightarrow{G}$. The only thing to observe here is that if $\C_i$ is at a vertex of the form $u^v_j$ ($j<t$) in $\overrightarrow{W}_t(\overrightarrow{G})$, then $\C_i$ is either at $u$ or $v$ in $\overrightarrow{G}$ and hence, $\C_i$ can always make the promised moves.

    \medskip
    \noindent\textit{Capture:} Finally, notice that when $\R$ gets captured in $\overrightarrow{W}_t(\overrightarrow{G})$ at an original vertex, then it gets captured in $\overrightarrow{G}$ as well. 
\end{proof}

Finally, the following theorem is implied by Lemma~\ref{L:wnDivision} and Lemma~\ref{L:swDivision}.

\begin{theorem}\label{T:wDivision}
Let $\overrightarrow{G}$ be an oriented graph. Then,
\begin{enumerate}[(i)]
     
\item $c_n(\overrightarrow{W}_t(\overrightarrow{G})) \geq c_n(\overrightarrow{G})$, 

\item  $c_w(\overrightarrow{W}_t(\overrightarrow{G})) \geq c_w(\overrightarrow{G})$,

\item  $c_s(\overrightarrow{W}_t(\overrightarrow{G})) \geq c_s(\overrightarrow{G})$. 
\end{enumerate}
\end{theorem}

\section{Computational Complexity}\label{com-compl}
In this section, we establish that assuming $P\neq NP$, for an oriented graph $\overrightarrow{G}$, there can be no polynomial time algorithm to decide whether $c_x(\overrightarrow{G}) = k$ where $x\in \{s,n,w\}$, even when $\overrightarrow{G}$ is restricted to be a $2$-degenerate oriented bipartite graph. These results are not very surprising since most pursuit-evasion games are indeed computationally difficult. We shall use the following well-known result on approximation hardness of \textsc{Minimum Dominating Set} (\textsc{MDS})~\cite{DSHard}. For a graph $G$,  $\gamma(G)$ denotes its \textit{domination number}, that is, the size of a minimum dominating set of $G$. 

\begin{proposition}[\cite{DSHard}]\label{P:AHard}
Unless $P=NP$, there is no polynomial time approximation algorithm that approximates \textsc{Minimum Dominating Set} with an approximation ratio 
$o(\log n)$. 
\end{proposition}

Fomin et al.~\cite{fomin} proved that \textsc{Cops and Robber} is NP-hard. They did so by providing a reduction from \textsc{MDS} on a graph $G$ to \textsc{Cops and Robber} on a graph $G'$. Moreover, in their construction, they had the following result.

\begin{proposition}[\cite{fomin}]\label{P:NHard}
A graph $G$ has a dominating set of size $k$ if and only if $G'$ is $k$-copwin.
\end{proposition}

Next, consider the graph $G'$. From $G'$, we get the graph $\overrightarrow{S}_2(G')$. Hence, we have the following corollary as a consequence of Proposition~\ref{P:NHard}, Theorem~\ref{th strong vs cn ub lb}, and Theorem~\ref{T:weakBound}.

\begin{corollary}\label{cor complexity}
For any graph $G$, we have 
$$\gamma(G) \leq c_s(\overrightarrow{S}_2(G')) \leq c_n(\overrightarrow{S}_2(G')) \leq c_w(\overrightarrow{S}_2(G')) \leq \gamma(G)+2.$$
\end{corollary}

\begin{proof}
From Theorem~\ref{th strong vs cn ub lb} and Theorem~\ref{T:weakBound}, it follows that $c(G')\leq c_s(\overrightarrow{S}_2(G')) \leq c_n(\overrightarrow{S}_2(G'))$  $\leq c_w(\overrightarrow{S}_2(G')) \leq c(G')+2$. Now, combining this with the fact that $\gamma(G) = c(G')$ (Proposition~\ref{P:NHard}), we have that $\gamma(G) \leq c_s(\overrightarrow{S}_2(G')) \leq c_n(\overrightarrow{S}_2(G')) \leq c_w(\overrightarrow{S}_2(G')) \leq \gamma(G)+2$.  
\end{proof}

Hence, if we can compute any of $c_s(\overrightarrow{S}_2(G'))$, $c_n(\overrightarrow{S}_2(G'))$, or $c_w(\overrightarrow{S}_2(G'))$ in polynomial time, then we have an additive $+2$ approximation for dominating set, which would imply that $P=NP$ (due to Proposition~\ref{P:AHard}). Therefore, we have the following theorem.

\begin{theorem}
Unless $P=NP$, for an oriented graph $\overrightarrow{G}$, there is no polynomial time algorithm to compute any of $c_s(\overrightarrow{G})$, $c_n(\overrightarrow{G})$, or $c_w(\overrightarrow{G})$ even if we restrict ourselves to  $2$-degenerate bipartite oriented graphs. 
\end{theorem}

\begin{proof}
    The proof follows from Corollary~\ref{cor complexity} and the observation that $\overrightarrow{S}_2(G)$ of any simple graph $G$ is bipartite and $2$-degenerate. 
\end{proof}

\section{Conclusions}\label{conclusion}
In this paper, we considered three variants of  the \textsc{Cops and Robber} game on oriented graphs, namely \textit{strong cop model}, \textit{normal cop model}, and \textit{weak cop model} with respect to subdivisions and retracts. We generalized and established various results on the relation between the cop numbers in these variants and the subdivisions and retracts. One interesting implication of our result concerning subdivisions was that computing the cop number in all these three models is computationally difficult. More specifically, unless $P=NP$, none of these problems can be solved in polynomial time on oriented graphs even when input is restricted to 2-degenerate bipartite graphs. We also remark that the idea of the proof of Theorem~\ref{th triangle_free} can also be used to establish that if we subdivide each edge of a triangle-free undirected graph an equal number of times, the cop number does not change.

We are still very far from a good understanding of the \textsc{Cops and Robber} game on oriented graphs. For example, the question of characterizing the strong-cop win graphs and normal-cop win graphs, our original motivation to study this problem, still remains open.  In an attempt to characterize strong-cop win oriented graphs, in Theorem~\ref{th strong-retract} we showed that a strong-retract of an oriented graph retains the strong-cop number. One natural question here can be to find out (all) the non-trivial examples of oriented stong-cop win graphs $\overrightarrow{G}$ which do not contain a strong-retract that is strong-cop win. 

Moreover, while the game is well-understood on several undirected graph classes like planar graphs~\cite{fromme1984game}, bounded-genus graphs~\cite{bowler}, geometric intersection graphs~\cite{ourString}, minor-free graphs~\cite{andreae}, we only know an upper bound of $O(\sqrt{n})$~\cite{loh2017cops} on the cop number of strongly connected planar directed graphs and do not know any lower bound better than $\omega(1)$~\cite{loh2017cops}. So, another interesting research direction is to explore the cop number of the (strongly connected) directed counterparts of the above-mentioned graph classes.

\bibliographystyle{abbrv}
\bibliography{main}

\end{document}